\documentclass[a4paper,10pt]{amsart} 
\usepackage{amsmath} 
\usepackage{amssymb} 
\usepackage{amsthm} 
\usepackage{mathtools,amsmath}
\usepackage{comment}

\usepackage{xcolor}
\usepackage{hyperref}
\hypersetup{
	colorlinks,
	linkcolor={red!50!black},
	citecolor={blue!50!black},
	urlcolor={blue!80!black},
	breaklinks=true
}

\usepackage[all]{xy}
\newtheorem{theorem}{\bf Theorem}[section]

\newtheorem{lemma}[theorem]{\bf Lemma}
\newtheorem{corollary}[theorem]{\bf Corollary}

\newtheorem{example}[theorem]{\bf Example}

\newcommand{\N}{\mathbb{N}}

\newcommand{\R}{\mathbb{R}}

\title{A family of four-variable expanders with quadratic growth }
\author[M. Makhul]{
	Mehdi Makhul$^{\ast}$}
\address[Mehdi Makhul]{Johann Radon Institute for Computational and Applied 
	Mathematics (RICAM), Austrian Academy of Sciences, Linz and Research Institute 
	for Symbolic Computation (RISC), Johannes Kepler University, Linz}
\email{mmakhul@risc.jku.at}
\thanks{$^\ast$ Supported by the Austrian Science Fund (FWF): W1214-N15, 
	Project DK9} 

\begin{document}

\begin{abstract}
	We prove that if $g(x,y)$ is a polynomial of constant degree $d$ that $y_2-y_1$ does not divide $g(x_1,y_1)-g(x_2,y_2)$, then for any finite set $A \subset  \R$ 
	\[ |X| \gg_d |A|^2, \quad \text{where} \
	X:=\left\{\frac{g(a_1,b_1)-g(a_2,b_2)}{b_2-b_1} :\, a_1,a_2,b_1,b_2 \in A \right\}.
	\]
	We will see this bound is also tight for some polynomial $g(x,y)$. 
\end{abstract}	

\maketitle

\section{Introduction}
Throughout this paper, when we write $X \gg Y$, this means that $X \ge cY$, for some absolute constant $c > 0$.

The sum set of a subset $A \subset \mathbb{R}$ is defined as $A+A:=\{a+b: a,b \in A \}$. The product set is defined in a similar way, $AA:=\{ab : a,b \in A \}$.

The Erd\H{o}s-Szemer\'edi \cite{Erdoes1983} conjecture states that, for all $\epsilon > 0$ and for any finite set $A \subset \N$,
\[
\text{max}\{|A+A|,|AA|\} \geq c(\epsilon) |A|^{2-\epsilon}.
\]
 It is natural to extend this conjecture for other settings (such as $\mathbb{R}$), and also to change the polynomials $F(x,y)=x+y$ and $F(x,y)=xy$ defining the sum and product sets to other polynomials or rational functions. In recent years much research has been done in this direction.

For many such functions, the images of sets are known to always grow. For example, the authors of \cite{Murphy2015} have studied several multivariable polynomials, including the function
\[
 G(x_1,x_2,x_3,x_4,x_5)=x_1(x_2+x_3+x_4+x_5).
\]
More precisely they showed that, for any finite set $A \subset \R$ ,
\[
|A(A+A+A+A)|\gg \frac{|A|^2}{\log|A|},
\]
where $A(A+A+A+A):=\{x_1(x_2+x_3+x_4+x_5) : x_i \in A \}$.

In \cite{Murphy2017}, the authors studied a more complicated function, namely
\[
H(x_1,x_2,x_3,x_4,x_5)=(x_1+x_2+x_3+x_4)^2+ \log x_5.
\]
They showed that, for any finite $A \subset \mathbb R$,
\[
|\left\{(a_1+a_2+a_3+a_4)^2+\log a_5 : a_i \in A \right\}| \gg \frac{|A|^2}{\log|A|}.
\]
In the same circle of ideas, \cite{Balog2015} investigated the rational function $F(x_1,x_2,x_3,x_4)=\frac{x_1+x_2}{x_3+x_4}$, showing that for any finite set $A \subset \R$, we have 
\[
|F(A,A,A,A)|\ge 2|A|^2-1.
\]
Our result is a generalization of the method of \cite[Corollary 3.1]{Murphy2015}, where they used the Szemer\'{e}di-Trotter Theorem to prove that for any finite set $A \subset \R$:
\[
 \left \lvert \frac{A-A}{A-A} \right \lvert  \gg |A|^2.
\]
A stronger version of this result, with a multiplicative constant $1$, follows from an earlier geometric result of Ungar \cite{Ungar1982}.

In this article we consider a certain class of rational functions of four variables. Suppose that $g(x,y)$ is a polynomial of two variables of degree $d$. Let 
 $$F(x_1,x_2,y_1,y_2)=\frac{g(x_1,y_1)-g(x_2,y_2)}{y_2 - y_1}$$
 be a four-variable rational function in terms of $x_1,x_2,y_1,y_2$. The main theorem of this paper is the following result concerning the growth of $F$.

\begin{theorem}{\label{main}}
	Suppose that $g(x,y)$ is a polynomial of degree $d$, that $y_2-y_1$ does not divide $g(x_1,y_1)-g(x_2,y_2)$, and that $A \subset  \R$ is a finite set. Then
	\[ |X| \gg_d |A|^2, \quad \text{where} \
	X:=\left\{\frac{g(a_1,b_1)-g(a_2,b_2)}{b_2-b_1} :\, a_1,a_2,b_1,b_2 \in A \right\}.
	\]
\end{theorem}

Notice that the following example shows that the condition that the denominator cannot be a divisor of the numerator is necessary.

\begin{example}
	Suppose that $g(x,y)=y^2$ and $A =\{1,2,\dots n\}$. Then
	\[
	X=\left\{\frac{b_1^2-b_2^2}{b_2-b_1}:\, b_1,b_2 \in A \right\} 
	\]
	equals $- \left\{ b_2+b_1 : b_i \in A \right\}$ and has cardinality $O(n)$.
\end{example}

Furthermore, notice that the condition rules out a degenerate case where the polynomial $g(x,y)$ does not depend on $x$.

On the other hand, it is known that for some polynomials $g$, the result of Theorem~\ref{main} is tight. For example, if we define $g(x_1,y_1)=x_1$ then Theorem \ref{main}  recovers the result of  \cite{Murphy2015} and \cite{Ungar1982}. This is known to be tight, since for the set $A=\{1,\dots,N\}$,
\[ \left | \frac{A-A}{A-A} \right | = O(N^2) . \]
However, we are not aware of any other polynomials $g$ for which the bound in Theorem \ref{main} is tight, and whether or not the bound can be improved for some particular $g$ is an interesting question.

Our main result has some similarities with a result of Raz, Sharir and Solymosi \cite{Raz2015} concerning the growth of two variable polynomials. Their result states that, if $F$ is a two variable polynomial with bounded degree, then for any $A, B \subset \mathbb R$ with $|A|=|B|=n$,
\[ |F(A,B)| \gg_d n^{4/3},\]
provided that $F$ satisfies a non-degeneracy condition. This condition states that $F$ cannot be of one of the following forms
\begin{enumerate}
	\item $F(u,v)=f(g(u)+h(v))$,
	\item$ F(u,v)=f(g(u)\cdot h(v))$.
\end{enumerate}
This result gave an improvement upon an earlier result of Elekes and Ronyai \cite{Elekes2000}.

\subsection{The Szemer\'{e}di-Trotter Theorem}

The essential ingredient used to prove our result is a corollary of the \emph{Szemer\'edi-Trotter} Theorem \cite{Szemeredi1983}, which gives a bound for the number of lines in the plane containing at least a fixed number of points $k$ from a given finite set, that is, the number of \textit{k-rich} lines.
\begin{theorem}[Szemer\'edi-Trotter]
	{\label{szT}}
	Suppose that $P$ is a set of $n$ points and $\mathcal{L}$ is a set of $m$ lines in $\mathbb R^2$. Then
	\begin{equation}{\label{szt2}}
	\mathcal{I}(P,\mathcal{L}) \ll n^{\frac{2}{3}}m^{\frac{2}{3}}+n+m.
	\end{equation}
\end{theorem}
\begin{corollary}{\label{k rich line}}
	Let $k, n\ge 2$ be natural numbers and fix $d\in \N$ such that $8d \leq k \leq d \sqrt{n}$. Let $\mathcal{L}$ be a set of $n$ lines in the plane, and let $t_{\ge k}$ denote the number of points in the plane contained in at least $k$ elements of $\mathcal{L}$, where each line appears with multiplicity at most $d$. Then 
	\[
	t_{\ge k}= O_d\left(\frac{n^2}{k^3}\right) .
	\]
\end{corollary}

\section{Main Results}

Suppose that $A,B \subset \mathbb{R}$ are finite, and $g(x_1,y_1)$ is a polynomial of degree $d$. We associate an element of $A \times B$ with a line via 
\[ 
 A \times B \ni (a,b) \; \longleftrightarrow \; l_{a,b} : y = bx - g(a,b).
\]
 Consider~$\mathcal{L}=\{\ell_{a,b} : a,b \in A \times B\}$ as a multi-set. Then~$\mathcal{L}$ is a set of~$|A||B|$ lines, such that each line appears at most~$d$ times. We also define the quantity 
 \[
 n(x,y) = \bigl|\bigl\{(a,b) \in A \times B \, : \, (x,y) \in l_{a,b} \bigr\}\bigr|,
 \] which is interpreted geometrically as the number of lines of $\mathcal{L}$ that pass through~$(x,y)$. 

\begin{lemma}{\label{lemma:energy}}
Suppose that $d \in \N$ is fixed. Suppose that $A,B,X \subset \mathbb{R}$ are finite and satisfy $|X| \le \frac{\lvert A \lvert\lvert B\lvert}{4d^2}$, with $0\notin X$. Then 
\begin{equation}{\label{eq:mine}}
\sum_{x\in X} \sum_{y} n^2(x,y) \ll |A|^{\frac{3}{2}}|B|^{\frac{3}{2}}|X|^{\frac{1}{2}}.
\end{equation}
\end{lemma} 

\begin{proof}
The amount of $t$-rich points is given by $$R_t:=\left\{ (x,y) \in \mathbb R^2: n(x,y) \ge t \right\}.$$ We first show that 
\[
|R_t| \ll \frac{|A|^2|B|^2}{t^3}.
\]
First, we bound $n(x,y)$ for a given point $(x,y)$. For fixed $b_0\in B$ we obtain a line with slope $b_0$ passing through $(x,y)$ and a one variable polynomial equation $g(a,b_0)$. Since each line is determined uniquely, by its slope and one point on it (for fixed $b_o$ and $(x,y)$ the equation $g(a,b_0) = 0$ has at most $d$ distinct solutions),  we have
\[
n(x,y) \le d\lvert B\lvert.
\]
With a similar argument for fixed $a \in A$ we obtain a univariate polynomial equation. Since each line is determined uniquely by its $y$-intercept and one point on it we have:
\[
n(x,y) \le d\lvert A \lvert.
\]
These together imply:
\[
n(x,y) \le d(\text{min}\{|A|,|B|\}) \le (d\lvert A\lvert d\lvert B\lvert )^{\frac{1}{2}}=d\lvert\mathcal{L}\lvert^{\frac{1}{2}}.
\]
This implies there are no points incident to more than $d\sqrt{|\mathcal{L}|}$ lines in $\mathcal{L}$, and by applying Corollary \ref{k rich line} we get:
\[
|R_t| \ll \frac{|\mathcal{L}|^2}{t^3} \le \frac{|A|^2|B|^2}{t^3}.
\]
Let $\Delta > 2d$ be an integer to be specified later. We have
\begin{equation}{\label{eq:mine1}}
\sum_{x \in X} \sum_y n^2(x,y) \le \sum_{x \in X}\sum_{n(x,y)\le \Delta}n^2(x,y)+\sum_{\substack{(x,y ) \\ n(x,y) > \Delta}}n^2(x,y).
\end{equation}
The first term is bounded by $\Delta |A||B||X|$, in fact
\begin{equation}{\label{eq:mine2}}
\sum_{x \in X}\sum_{n(x,y)\le \Delta}n^2(x,y) \le \Delta\sum_{x \in X}\sum_y n(x,y)=\Delta|A||B|\sum_{x \in X}1=\Delta|A||B||X|.
\end{equation}

For second term we have: 
\begin{equation}{\label{eq:mine3}}
\begin{aligned}
\sum_{\substack{(x,y ) \\ n(x,y) > \Delta}}n^2(x,y) &= \sum_{j \ge 1} \;\sum_{2^{j-1}\Delta\le n(x,y) \le 2^j \Delta } n^2(x,y) \ll \\ &\ll \sum_{j \ge 1}\frac{|A|^2|B|^2}{(2^j\Delta )^3} \cdot (2^j \Delta)^2=\frac{|A|^2|B|^2}{\Delta}\sum_{j \ge 1} \frac{1}{2^j}=\frac{|A|^2|B|^2}{\Delta}.
\end{aligned}
\end{equation}
For an optimal choice, set the parameter $\Delta= \Big\lceil \frac{(|A\|B|)^{1/2}}{|X|^{1/2}}\Big\rceil > 2d$. Combining the bounds from \eqref{eq:mine1} and \eqref{eq:mine2} and \eqref{eq:mine3}, it follows that
\[
\sum_{x} \sum_{y} n^2(x,y) \ll |A|^{\frac{3}{2}}|B|^{\frac{3}{2}}|X|^{\frac{1}{2}} \, . \qedhere
\]
\end{proof}

\textit{Proof of Theorem \ref{main}}. Consider:
\begin{multline*}
|X| = \left|\left\{(x,a_1,a_2,b_1,b_2): x=\frac{g(a_1,b_1)-g(a_2,b_2)}{b_1-b_2} , a_i,b_i \in A\right\}\right| \\[2ex]
=\bigl|\bigl\{(x,a_1,a_2,b_1,b_2) \, : \, b_1x-g(a_1,b_1)=b_2x-g(a_2,b_2) \bigr\}\bigr|= \\[2ex]
\sum_{x} \sum_{y} n^2(x,y) \ll |A|^3|X|^{\frac{1}{2}} \, .
\end{multline*}

On the other hand, $|X| \ge |A|^4$. Hence we obtain: 
\begin{center}
 \hfill $|A|^4 \ll |A|^3|X|^{\frac{1}{2}}, \quad \text{hence} \quad |X| \gg |A|^2 \,.$  \hfill$\square$
\end{center}

\begin{corollary}
Suppose that $P= A \times A$ is a set of $\lvert A \lvert^2$ points. Let $l$ be the $y$-axis. Suppose that $B(P)$ is the set of all bisectors determined by $P$. Then $\lvert B \cap l \lvert \gg \lvert A \lvert^2$. 
\end{corollary} 
\begin{proof}
By a simple calculation we can see that the equation of the bisector determined by two points $(x_1,y_1)$ and $(x_2,y_2)$ in the $s,t$ plane is:
\[
s=\frac{2(x_1-x_2)t+(x_2^2-x_1^2)+(y_2^2-y_1^2)}{2(y_2-y_1)}.
\] 
Inserting $t=0$, the hitting point has coordinate
\[
\left(0,\frac{(x_2^2-x_1^2)+(y_2^2-y_1^2)}{2(y_2-y_1)}\right).
\]
Setting $g(x,y)=-2(x^2-y^2)$, we obtain the result by Theorem \ref{main}.  
\end{proof}
As we mentioned, this bound is tight for some polynomials, for instance $g(x,y)=x$. However, we expect that if $F(x_1,x_2,y_1,y_2)$ is a generic rational function satisfying the condition of the Theorem \ref{main} we have $\lvert X \lvert \gg \lvert A \lvert^3$.

\bigskip
\bigskip
\noindent {\Large \textbf{Acknowledgements}}.
I would like to thank Oliver Roche-Newton for bringing this problem to my attention and for several helpful conversations.

\providecommand{\bysame}{\leavevmode\hbox to3em{\hrulefill}\thinspace}
\providecommand{\MR}{\relax\ifhmode\unskip\space\fi MR }
\providecommand{\MRhref}[2]{%
	\href{http://www.ams.org/mathscinet-getitem?mr=#1}{#2}
}
\providecommand{\href}[2]{#2}


\begin{thebibliography}{MRNS17}
	
	\bibitem[BRN15]{Balog2015}
	Antal Balog and Oliver Roche-Newton, \emph{New sum-product estimates for real
		and complex numbers}, Discrete Comput. Geom. \textbf{53} (2015), no.~4,
	825--846. \MR{3341581}
	
	\bibitem[ER00]{Elekes2000}
	Gy\"orgy Elekes and Lajos R\'onyai, \emph{A combinatorial problem on
		polynomials and rational functions}, J. Combin. Theory Ser. A \textbf{89}
	(2000), no.~1, 1--20. \MR{1736139}
	
	\bibitem[ES83]{Erdoes1983}
	Paul Erd\"{o}s and Endre Szemer\'edi, \emph{On sums and products of integers},
	Studies in pure mathematics, Birkh\"{a}user, Basel, 1983, pp.~213--218.
	\MR{820223}
	
	\bibitem[MRNS15]{Murphy2015}
	Brendan Murphy, Oliver Roche-Newton, and Ilya Shkredov, \emph{Variations on the
		sum-product problem}, SIAM J. Discrete Math. \textbf{29} (2015), no.~1,
	514--540. \MR{3323540}
	
	\bibitem[MRNS17]{Murphy2017}
	Brendan Murphy, Oliver Roche-Newton, and Ilya~D. Shkredov, \emph{Variations on
		the {S}um-{P}roduct {P}roblem {II}}, SIAM J. Discrete Math. \textbf{31}
	(2017), no.~3, 1878--1894. \MR{3691216}
	
	\bibitem[RSS15]{Raz2015}
	Orit~E. Raz, Micha Sharir, and J\'ozsef Solymosi, \emph{On triple intersections
		of three families of unit circles}, Discrete Comput. Geom. \textbf{54}
	(2015), no.~4, 930--953. \MR{3416906}
	
	\bibitem[ST83]{Szemeredi1983}
	Endre Szemer\'edi and William~T. Trotter, Jr., \emph{Extremal problems in
		discrete geometry}, Combinatorica \textbf{3} (1983), no.~3-4, 381--392.
	\MR{729791}
	
	\bibitem[Ung82]{Ungar1982}
	Peter Ungar, \emph{{$2N$}\ noncollinear points determine at least {$2N$}\
		directions}, J. Combin. Theory Ser. A \textbf{33} (1982), no.~3, 343--347.
	\MR{676751}
	
\end{thebibliography}
\end{document}